\documentclass[12pt, twoside, leqno]{article}
\usepackage{amsmath,amsthm}
\usepackage{amssymb,latexsym}
\usepackage{enumerate}

\newtheorem{theorem}{Theorem}[section]
\newtheorem{lemma}[theorem]{Lemma}

\newtheorem{proposition}[theorem]{Proposition}

\newtheorem{question}[theorem]{Question}

\newtheorem{example}[theorem]{Example}

\numberwithin{equation}{section}

\begin{document}

\baselineskip=17pt

\title{On weakly Gibson $F_\sigma$-measurable mappings}

\author{Olena Karlova\\
Chernivtsi National University,\\
                    Department of Mathematical Analysis,\\
                    58012 Chernivtsi, Ukraine\\
E-mail: maslenizza.ua@gmail.com
\and
Volodymyr Mykhaylyuk\\
Chernivtsi National University,\\
                    Department of Mathematical Analysis,\\
                    58012 Chernivtsi, Ukraine\\
E-mail: vmykhaylyuk@ukr.net}

\date{}

\maketitle


\renewcommand{\thefootnote}{}



\renewcommand{\thefootnote}{\arabic{footnote}}
\setcounter{footnote}{0}

\begin{abstract}
   A function $f:X\to Y$ between topological spaces is said to be a {\it weakly Gibson function} if $f(\overline{U})\subseteq \overline{f(U)}$ for any open connected set \mbox{$U\subseteq X$}. We prove that if $X$ is a locally connected hereditarily Baire space and $Y$ is a $T_1$-space then an $F_\sigma$-measurable mapping $f:X\to Y$ is weakly Gibson if and only if for any connected set $C\subseteq X$ with the dense connected interior   
    the image $f(C)$ is connected. Moreover, we show that each weakly Gibson $F_\sigma$-measurable mapping $f:\mathbb R^n\to Y$, where $Y$ is a $T_1$-space, has a connected graph.
\end{abstract}

\section{Introduction}

The classical theorem of Kuratowski and Sierpi\'{n}ski \cite{KuSerp} states that any Darboux Baire-one function $f:\mathbb R\to\mathbb R$ has a connected graph.

In 2010 K.~Kellum~\cite{Kellum} introduced Gibson and weak Gibson properties for a mapping $f$ between topological spaces $X$ and $Y$. He calls $f$ (weakly) Gibson if $f(\overline{U})\subseteq \overline{f(U)}$ for an arbitrary open (and connected) set $U\subseteq X$. Since every Darboux function has the weak Gibson property~\cite{KaMykh}, it is naturally to ask whether the theorem of Kuratiwski--Sierpi\'{n}ski remains valid if we replace the Darboux property by the weak Gibson property? It was shown in \cite{KaMykh} that any weakly Gibson barely continuous mapping (in the sense that for each non-empty closed subspace $F\subseteq X$ the restriction $f|_F$ has a continuity point) defined on a connected and locally connected space $X$ and with values in a topological space  $Y$ has a connected graph. It is find out that the condition of barely continuity in the above mentioned result from \cite{KaMykh} is not necessary  (see Example~\ref{multivalued}).

In this paper we consider weakly Gibson mappings $f:X\to Y$ which are $F_\sigma$-measurable, i.e. the preimage $f^{-1}(V)$ of an open set $V\subseteq Y$ is an $F_\sigma$-set in $X$. Note that in the case when $Y$ is a perfectly normal space, every Baire-one mapping $f:X\to Y$ is  $F_\sigma$-measurable (see for instance \cite[p.~394]{Ku1}). In Section~\ref{sec:sets} we introduce the notions of $\mathcal G$-closed and $\mathcal W$-closed sets and prove that the Euclidean space $\mathbb R^n$ cannot be written as a union of two non-empty disjoint $F_\sigma$ and $G_\delta$ $\mathcal W$-closed subsets as well as a connected and locally connected hereditarily Baire space cannot be written as a union of two non-empty disjoint $F_\sigma$ and $G_\delta$ $\mathcal G$-closed subsets. Using these facts in Section~\ref{sec:applications} we prove that each $F_\sigma$-measurable mapping $f$ between a locally connected hereditarily Baire space $X$ and a $T_1$-space $Y$ is weakly Gibson if and only if for any connected set $C\subseteq X$ with the dense connected interior the image $f(C)$ is connected. This generalizes the result of M. Evans and P. Humke  \cite{EvansHumke} who proved the similar theorem for $X=\mathbb R^n$ and $Y=\mathbb R$.
We prove also that each weakly Gibson $F_\sigma$-measurable mapping $f:\mathbb R^n\to Y$, where $Y$ is a $T_1$-space, has a connected graph.

\section{$\mathcal A$-closed sets and their properties}\label{sec:sets}
Let $X$ be a topological space and let
\begin{itemize}
  \item $\mathcal T(X)$ be the system of all open subsets of $X$,

  \item $\mathcal C(X)$ be the system of all connected subsets of $X$,

  \item $\mathcal G(X)$ be the system of all connected open subsets of $X$,

  \item $\mathcal W(X)$ be the system of all open convex subsets of a topological vector space $X$.
\end{itemize}

Let $\mathcal A(X)$ be a system of subsets of $X$. A subset $E\subseteq X$ is called {\it closed with respect to $\mathcal A(X)$} or, briefly, {\it $\mathcal A$-closed} if for any $A\in \mathcal A(X)$ with $A\subseteq E$ we have $\overline{A}\subseteq E$.

\begin{proposition}
  Let $X$ be a connected and locally connected space and $U$ be an open $\mathcal G$-closed subset of $X$. Then $U=\emptyset$ or $U=X$.
\end{proposition}

\begin{proof}
  Consider a component $C$ of $U$. The locally connectedness of $U$ implies that $C$ is clopen in $U$, consequently, $C$ is open in $X$. Since $U$ is $\mathcal G$-closed, $\overline{C}\subseteq U$. Therefore, $\overline{C}=C$ provided $C$ is a component. Hence, $C$ is clopen in a connected space $X$. Therefore, $C=\emptyset$ or $C=X$. Since  $U$ is a union of all components,  $U=\emptyset$ or $U=X$.
\end{proof}

We need the following auxiliary fact.

\begin{lemma}\label{lemma:3:10}{\rm {\cite[p.~136]{Ku1}}}
  Let $A$ and $B$ be subsets of a topological space $X$ such that $A$ is connected and  $A\cap B\ne\emptyset\ne A\setminus B$. Then $A\cap {\rm fr B}\ne\emptyset$.
\end{lemma}

For a point $x_0$ of a normed space $X$ and for $\varepsilon>0$ by $B(x_0,\varepsilon)$ /$B[x_0,\varepsilon]$/ we denote an open /closed/ ball with the center at $x_0$ and radius $\varepsilon$.

If a  subset of a topological space is simultaneously $F_\sigma$ and $G_\delta$, then it is said to be {\it ambiguous}.

\begin{theorem}\label{th:main}
  Let $X$ be a hereditarily Baire space, $X_1$ and $X_2$ are ambiguous disjoint $\mathcal A$-closed subsets of $X$ such that $X=X_1\cup X_2$. If
\begin{enumerate}
\item $X$ is a connected and locally connected space and $\mathcal A(X)=\mathcal G(X)$,  or

\item $X=\mathbb R^n$, $n\ge 1$, and $\mathcal A(X)=\mathcal W(X)$,
\end{enumerate}
then $X_1=X$ or $X_2=X$.
\end{theorem}

\begin{proof}
To obtain a contradiction, suppose that $X_1\ne X$ and $X_2\ne X$.  Let $F=\overline{X}_1\cap\overline{X}_2$. Since $X$ is connected, $F\ne\emptyset$. We show that $X_1\cap F$ is dense in $F$. Conversely, choose a point $x_0\in F$ and an open neighborhood $U$ of $x_0$ in  $X$ such that
$$U\cap F\subseteq X_2.$$
Then $x_0\in \overline{X}_1\cap X_2$.

1). Since $X$ is locally connected, we may assume that $U$ is connected. Note that $U\cap X_1\ne\emptyset$ and take $a\in U\cap X_1$. Then $a\not\in \overline{X}_2$. Let $G$ be a component of $X\setminus \overline{X}_2$ which contains $a$. Then $G$ is open in $X$. Remark that $U\cap G\ne\emptyset\ne U\setminus G$. Lemma~\ref{lemma:3:10} implies that $U\cap {\rm fr}G\ne\emptyset$. Since $G$ is closed in $X\setminus \overline{X}_2$,  ${\rm fr}G\subseteq \overline{X}_2$. Moreover, $G\subseteq X_1$. Therefore, ${\rm fr}G\subseteq F$. Choose $b\in U\cap {\rm fr} G$. Then $b\in X_2$. Since  $X_1$ is $\mathcal G$-closed, $b\in \overline{G}\subseteq X_1$, which is impossible.

2). We may suppose that $U=B(x_0,\varepsilon)$. Take an arbitrary $a\in B(x_0,{\varepsilon/2})\cap X_1$. Let
$$
R=\sup\{r: B(a,r)\subseteq X_1\}.
$$
Note that $R\le\varepsilon/2$, since $x_0\in X_2$.  We have
$$
d(x,x_0)\le d(x,a)+d(a,x_0)< R+\varepsilon/2<\varepsilon/2+\varepsilon/2=\varepsilon
$$
for all $x\in B[a,R]$. Hence, $ B[a,R]\subseteq  U$.
It is not hard to verify that $B[a,R]\cap \overline X_2\ne\emptyset$, provided $B[a,R]$ is compact. Therefore, there is $b\in B[a,R]\cap \overline X_2$. Since $B(a,R)$ is open and convex and $X_1$ is $\mathcal W$-closed, $b\in X_1$. But $b\in U\cap F$, which implies that $b\in X_2$. Thus, $b\in X_1\cap X_2$ which is impossible.

Hence,  $X_1\cap F$ is dense in $F$. It can be proved similarly that $X_2\cap F$ is dense in $F$. Then $X_1\cap F$ and $X_2\cap F$ are disjoint dense $G_\delta$-subsets of a Baire space $F$, which implies a contradiction. Therefore, $X_1=X$ or $X_2=X$.
\end{proof}

\section{Applications of $\mathcal A$-closed sets}\label{sec:applications}

We say that a mapping $f:X\to Y$ has  {\it a Gibson property with respect to a system $\mathcal A(X)$}, or {\it $f$ is an $\mathcal A$-Gibson} if for any  $A\in \mathcal A(X)$ we have
$$
f(\overline{A})\subseteq\overline{f(A)}.
$$
It $\mathcal A(X)=\mathcal T(X)$ then  $f$ is said to be {\it a Gibson mapping}, and if $\mathcal A(X)=\mathcal G(X)$ then $f$ {\it is a weakly Gibson mapping} (see \cite{Kellum}).

A mapping $f:X\to Y$ is {\it strongly Gibson with respect to a system  $\mathcal A(X)$}, or {\it $f$ is strongly $\mathcal A$-Gibson} if for any $x\in X$ and $A\in\mathcal A(X)$ such that $x\in\overline A$ we have
$$
f(x)\in\overline{f(A\cap U)}
$$
for an arbitrary neighborhood $U$ of  $x$ in $X$.

\begin{theorem}\label{th:strong}
  Let $X$ be a topological space, $Y$ a $T_1$-space and let $f:X\to Y$ be a mapping such that for any connected set $C\subseteq X$  with the dense connected interior the set $f(C)$ is connected.  Then $f$ is a weakly Gibson mapping.

  If, moreover, $X$ is a locally convex space then $f$ has the strong Gibson property with respect to the system $\mathcal W(X)$.
\end{theorem}

\begin{proof}
  Fix an arbitrary open connected set $U\subseteq X$, a point $x_0\in\overline U$ and an open neighborhood $V$ of $f(x_0)$ in $Y$. Denote $C=U\cup\{x_0\}$. Then the inclusions $U\subseteq C\subseteq\overline{U}$ imply that $f(C)$ is a connected set. Assume $f(U)\cap V=\emptyset$. Then
$$
f(C)=f(U\cup\{x_0\})=f(U)\cup \{f(x_0)\}\subseteq (Y\setminus V)\cup \{f(x_0)\},
$$
which contradicts to the connectedness of $f(C)$.

Now let $X$ be a locally convex space. Fix a set $G\in\mathcal W(X)$, a point $x_0\in \overline G$, an open convex neighborhood $W$ of $x_0$ in $X$ and an open neighborhood $V$ of $f(x_0)$ in $Y$. Denote $U=W\cap G$. Clearly, $U\in\mathcal G(X)$. The rest of the proof runs as before.
\end{proof}

The converse proposition is true for $F_\sigma$-measurable mappings defined on a locally connected hereditarily Baire space.

\begin{theorem}\label{th:wGconn}
Let $X$ be a locally connected hereditarily Baire space, $Y$ a topological space and let $f:X\to Y$ be a weakly Gibson $F_\sigma$-measurable mapping.
Then for any  connected set $C\subseteq X$ with the dense connected interior the set $f(C)$ is connected.
 \end{theorem}

\begin{proof}
Let $C\in\mathcal C(X)$, $U={\rm int\,}C$ and $C\subseteq \overline{U}$.

We first prove that $f(U)$ is a connected set. Suppose, contrary to our claim, that $f(U)= W_1\cup W_2$, where $W_1$ and $W_2$ are non-empty disjoint open subsets of  $f(U)$. Set $g=f|_U$. Evidently, $g:U\to f(U)$ is a weakly Gibson $F_\sigma$-measurable mapping. Let $A_i=g^{-1}(W_i)$ for $i=1,2$. Then every set $A_i$ is $\mathcal G$-closed in $U$, provided $g$ is weakly Gibson. Moreover, every $A_i$ is ambiguous set in $U$, $U=A_1\cup A_2$ and \mbox{$A_1\cap A_2=\emptyset$}. Taking into account that $U$ is a hereditarily Baire connected and locally connected space, we obtain that $A_1=U$ or $A_2=U$ according to Theorem~\ref{th:main}(1).
Then $W_1=\emptyset$ or $W_2=\emptyset$, a contradiction. Therefore, $f(U)$ is a connected set.

Since $f$ is weakly Gibson, $f(U)\subseteq f(C)\subseteq f(\overline{U})\subseteq \overline{f(U)}$. Consequently,  the set $f(C)$ is connected.
\end{proof}

For a mapping $f:X\to Y$ we define $\gamma_f:X\to X\times Y$,
$$
\gamma_f(x)=(x,f(x)).
$$
Remark that if $X$ is a connected and locally connected hereditarily Baire space and $\gamma_f$ is an $F_\sigma$-measurable weakly Gibson mapping then Theorem~\ref{th:wGconn} implies that $f$ has a connected graph $\Gamma$, provided $\Gamma=\gamma_f(X)$. It is not hard to prove that $\gamma_f$ remains to be weakly Gibson for any weakly Gibson mapping $f:\mathbb R\to\mathbb R$. But Example~\ref{ex:notWeaklyGibson} shows that $\gamma_f$ need not be weakly Gibson for a weakly Gibson $F_\sigma$-measurable mapping $f:\mathbb R^2\to\mathbb R$.

\begin{theorem}\label{th:MainMain}
  Let $X=\mathbb R^n$ with  $n\ge 1$ and let $Y$ be a $T_1$-space. If $f:X\to Y$ is a weakly Gibson $F_\sigma$-measurable mapping then $f$ has a connected graph.
\end{theorem}

\begin{proof}
  We first observe that by Theorem~\ref{th:wGconn} for any $U\in\mathcal G(X)$ and for any $C$ with $U\subseteq C\subseteq \overline {U}$ the set $f(C)$ is connected. Then $f$ has the strong Gibson property with respect to the system $\mathcal W(X)$ according to Theorem~\ref{th:strong}. It is easy to see that $\gamma_f$ is also $\mathcal W$-strongly Gibson.

     We show that $\gamma_f:X\to X\times Y$ is $F_\sigma$-measurable. Let $\{B_k:k\in\mathbb N\}$ be a base of open sets in $X$ and $W$ be an arbitrary open set in  $X\times Y$. Put $$V_k=\bigcup\{V: V \mbox{ is open in  } Y \mbox{ and } B_k\times V\subseteq W\}.$$ Then  $W=\bigcup\limits_{k=1}^\infty (B_k\times V_k)$. Since $\gamma_f^{-1}(W)=\bigcup\limits_{k=1}^\infty (B_k\cap f^{-1}(V_k))$, $\gamma_f^{-1}(W)$ is an $F_\sigma$-subset of $X$.

 Now assume that $Y_0=\gamma_f(X)$ is not connected and choose open disjoint non-empty subsets $W_1$ and $W_2$ of $Y_0$ such that $Y_0=W_1\cup W_2$. Let $X_i=\gamma_f^{-1}(W_i)$ for $i=1,2$. It is easy to check that $X_1$ and $X_2$ are $\mathcal W$-closed ambiguous subsets of  $X$. Moreover, $X_1\cap X_2=\emptyset$ and $X=X_1\cup X_2$. Then $X_1=X$ or $X_2=X$ by Theorem~\ref{th:main}~(2). Consequently, $W_1=\emptyset$ or $W_2=\emptyset$, a contradiction.
\end{proof}

The following question is open.

\begin{question}
  Let $X$ be a normed space, $Y$ a $T_1$-space and let $f:X\to Y$ be a weakly Gibson $F_\sigma$-measurable mapping. Is the graph of $f$ a connected set?
\end{question}

\section{Examples}

Our first example shows that the class of all $F_\sigma$-measurable Darboux mappings is strictly wider than the class of all Baire-one Darboux mappings.

\begin{example}
There exist a connected subset $Y\subseteq\mathbb R^2$ and an $F_\sigma$-measurable Darboux function $f:\mathbb R\to Y$ which is not a Baire-one function.
\end{example}

\begin{proof}
  Let $\mathbb Q=\{r_n:n\in\mathbb N\}$ be the set of all rational numbers. For every $n\in\mathbb N$ we consider the function $\varphi_n:\mathbb R\to\mathbb R$,
  $$
  \varphi_n(x)=\left\{\begin{array}{ll}
                        \sin\frac{1}{x-r_n}, & x\ne r_n, \\
                        0, & x=r_n.
                      \end{array}
  \right.
  $$
Define the function $g:\mathbb R\to\mathbb R$,
  $$
  g(x)=\sum\limits_{n=1}^\infty \frac{1}{2^n}\varphi_n(x).
  $$
Let
$$
  Y=\{(x,y)\in\mathbb R^2:y=g(x)\}\quad\mbox{and}\quad f=\gamma_g.
$$
   Observe that for every  $n$ the function $g_n(x)=\sum\limits_{k=1}^n \frac{1}{2^k}\varphi_k(x)$ is a Baire-one Darboux function. Since the sequence $(g_n)_{n=1}^\infty$ is uniformly convergent to $g$ on $\mathbb R$, $g$ is a Baire-one Darboux function~\cite[Theorem 3.4]{Bru}.  Consequently, the graph of $g|_C$ is connected for every connected subset $C\subseteq\mathbb R$ according to \cite[Theorem 1.1]{Bru}. Therefore, $f:\mathbb R\to Y$ is a Darboux function. Moreover, $f:\mathbb R\to \mathbb R^2$ is a Baire-one mapping, which implies that $f:\mathbb R\to Y$ is $F_\sigma$-measurable.

  Note that the space $Y$ is punctiform (i.e., $Y$ does not contain any continuum of cardinality larger than one), since $g$ is discontinuous on everywhere dense set $\mathbb Q$ (see \cite{KuSerp}). Then each continuous mapping between $\mathbb R$ and $Y$ is constant. Therefore, $f:\mathbb R\to Y$ is not a Baire-one mapping.
\end{proof}

\begin{example}\label{ex:notWeaklyGibson}
  For all $(x,y)\in\mathbb R^2$ define
  $$
   f(x,y)=\left\{\begin{array}{ll}
                   \sin\frac 1x, & x>0,\\
                   1, & x\le 0.
                 \end{array}
   \right.
  $$
Then $f:\mathbb R^2\to \mathbb R$ is an $F_\sigma$-measurable weakly Gibson function, but $\gamma_f$  is not weakly Gibson.
\end{example}

\begin{proof}
  Show that $f$ is weakly Gibson. It is sufficient to check that $f$ is weakly Gibson at each point of the set $\{0\}\times \mathbb R$. Fix $y_0\in\mathbb R$ and an open connected set $U\subseteq \mathbb R^2$ such that $p_0=(0,y_0)\in\overline{U}\setminus U$. Take an arbitrary neighborhood $V$ of $f(p_0)$ in $\mathbb R$. Clearly, $f(p)\in V$ for all $p\in U\cap ((-\infty,0]\times \mathbb R)$. Consider the case $U\subseteq (0,+\infty)\times\mathbb R$. Since $p_0\in\overline{U}$ and $U$ is connected, there exists $n\in\mathbb N$ such that $U\cap(\{\frac{1}{\pi/2+2\pi n}\}\times\mathbb R)\ne\emptyset$. Let $y\in\mathbb R$ with $p=(\frac{1}{\pi/2+2\pi n},y)\in U$. Then $f(p)=1$ and  $f(p)\in V$. Hence, $f$ is weakly Gibson.

   Consider open connected set $U=\{(x,y)\in\mathbb R^2: x>0\,\,\mbox{and}\,\, |y-\sin\frac 1x|<x\}$ and let $C=U\cup\{(0,0)\}$. Then $U\subseteq C\subseteq\overline{U}$. Note that $\gamma_f:\mathbb R^2\to\mathbb R^3$ is $F_\sigma$-measurable. One easily checks that  $\gamma_f(C)$ is not connected. Therefore, $\gamma_f$ is not weakly Gibson by Theorem~\ref{th:wGconn}.
\end{proof}

Finally, we give an example of a space $Y$ and an $F_\sigma$-measurable Darboux mapping $f:\mathbb R\to Y$ which is not barely continuous.

We need first some definitions and auxiliary facts. For a topological space $Y$ by ${\mathcal F}(Y)$ we denote the space of all non-empty closed subsets of $Y$ equipped with the Vietoris topology.
A multivalued mapping $F:X\to Y$ is said to be {\it upper (lower) continuous at $x_0\in X$} if for any open set  $V$ in $Y$ such that $F(x_0)\subseteq V$ ($F(x_0)\cap V\ne\emptyset$) there exists a neighborhood $U$ of $x_0$ in $X$ such that for every  $x\in U$ we have $F(x)\subseteq V$ ($F(x)\cap V\ne\emptyset$). A multivalued mapping $f$ which is upper and lower continuous at $x_0$ is called {\it continuous at $x_0$}.

\begin{lemma}\label{pr:1.2}
 There exists a continuous mapping $f_0:\mathbb R\to {\mathcal F}(\mathbb R)$ such that
for all $x\in [0,1]$ and $p\in P=\{\frac{1}{n}:n\in\mathbb N\}\cup \{0\}$  there are $n_p\in\mathbb N$,  strictly increasing unbounded sequence  $(v_n)_{n\geq n_p}$ of reals  $v_n>0$ and strictly decreasing unbounded sequence  $(u_n)_{n\geq n_p}$  of reals  $u_n<0$  such that   $$f_0(u_n)=f_0(v_n)=\{p\}\cup\bigcup\limits_{k=1}^{n}[k,k+x]\cup\bigcup\limits_{k>n}\{k\}$$ for all $n\geq n_p$.
\end{lemma}

\begin{proof}
Let $P=\{p_n:n\in\mathbb N\}$. Choose a continuous function $\varphi_0:\mathbb R\to [0,1]$ with $\varphi_0(x)=p_k$ if $|x|\in[n+\frac{2k-1}{2n}, n+\frac{k}{n}]$, where $n\in\mathbb N$ and $k=1,\dots,n$. For every $n\in\mathbb N$ define a continuous function $\varphi_n:\mathbb R\to [n,n+1]$,
$$\varphi_n(x)=\left\{\begin{array}{ll}
                         n, & |x|\leq n,\\
                         n+\sin(4\pi k|x|), & |x|\in (k,k+1], k\geq n.
                       \end{array}
 \right.$$
Let $f_0(x)=\{\varphi_0(x)\}\cup\bigcup\limits _{n=1}^{\infty}[n,\varphi_n(x)]$. Since all the functions $\varphi_n$ are continuous and $\varphi_k(x)=k$ for $x\in [-n,n]$ and $k\geq n$, $f_0$ is continuous.

Fix $p=p_n\in P$ and $x\in[0,1]$. Denote $n_p=n$. For all $k\geq n$ choose $v_k\in [k+\frac{2n-1}{2k}, k+\frac{n}{k}]$ such that $\sin(4\pi k v_k)=x$. Then for every  $k\geq n$ we have $\varphi_0(v_k)=p$, $\varphi_1(v_k)=1+x$,\dots, $\varphi_k(v_k)=k+x$  and $\varphi_i(v_k)=i$ for $i>k$, i.e., the sequence $(v_k)_{k\geq n}$ satisfies the condition of the lemma. It remains to set $u_k=-v_k$ for all $k\in\mathbb N$.
\end{proof}

\begin{example}\label{multivalued}
 There exists a Baire-one $F_{\sigma}$-measurable Darboux mapping $f:\mathbb R\to {\mathcal F}(\mathbb R)$ such that the restriction $f|_C$ of $f$ on the Cantor set $C\subseteq \mathbb R$ is everywhere discontinuous and $f(\mathbb R)$ is hereditarily Lindel\"{o}f (in particular, $f(\mathbb R)$ is perfectly normal).
\end{example}

\begin{proof} Let $\mathbb R\setminus C=\bigcup\limits _{n=1}^{\infty}I_n$, where $I_n=(a_n,b_n)$. Set $A=\{a_n:n\in\mathbb N\}$ and $B=C\setminus A$. For every $n\in\mathbb N$ we choose a homeomorphism $\psi_n:I_n\to\mathbb R$. Define
$$f(x)=\left\{\begin{array}{lll}
                         f_0(\psi_n(x)), & x\in I_n,\\
                         \{\frac{1}{n}\}\cup\bigcup\limits_{k=1}^{\infty}[k,k+x], & n\in\mathbb N, x=a_n,\\
                         \{0\}\cup\bigcup\limits_{k=1}^{\infty}[k,k+x], & x\in B,
                       \end{array}
 \right.$$
 where $f_0$ is the function from Lemma~\ref{pr:1.2}.

Show that $f$ is a Baire-one mapping. For every $n\in\mathbb N$ applying Lemma~\ref{pr:1.2}, we find a number $m_n$, strictly increasing sequence $(v_k^{(n)})_{k\geq m_n}$ of $v_k^{(n)}\in(\frac{a_n+b_n}{2},b_n)$ and strictly decreasing sequence $(u_k^{(n)})_{k\geq m_n}$ of $u_k^{(n)}\in(\frac{a_n,a_n+b_n}{2})$ such that $$f_0(\psi_n(u_k^{(n)}))=\{\frac{1}{n}\}\cup\bigcup\limits_{i=1}^{k}[i,i+a_n]\cup\bigcup\limits_{i>k}\{i\}$$ and $$f_0(\psi_n(v_k^{(n)}))=\{0\}\cup\bigcup\limits_{i=1}^{k}[i,i+a_n]\cup\bigcup\limits_{i>k}\{i\}$$ for all $i\geq m_n$.

For every $n\in\mathbb N$ denote $M_n=\{k\leq n:m_k\leq n\}$. Clearly, $M_n\subseteq M_{n+1}$ for all $n$ and $\mathbb N=\bigcup\limits_{n=1}^{\infty}M_n$. Choose a sequence of continuous functions $g_n:\mathbb R\to [0,1]$ which is pointwise convergent to the function
$$g(x)=\left\{\begin{array}{ll}
                         0, & x\in \mathbb R\setminus A,\\
                         \frac{1}{n}, & n\in\mathbb N, x=a_n.
                       \end{array}
 \right.$$
Without loss of generality, we assume that $g_n(u_k^{(n)})=\frac{1}{n}$ and $g_n(v_k^{(n)})=0$ if $n\in M_k$. Now for every $k\in\mathbb N$ define
$$f_k(x)=\left\{\begin{array}{ll}
                         f_0(\psi_n(x)), & x\in [u_k^{(n)},v_k^{(n)}], \, n\in M_k,\\
                         \{g_k(x)\}\cup\bigcup\limits_{i=1}^{k}[i,i+x]\cup\bigcup\limits_{i>k}\{i\}, & x\in \mathbb R \setminus\left(\bigcup\limits_{n\in M_k}[u_k^{(n)},v_k^{(n)}]\right).
                       \end{array}
 \right.$$
It is easy to see that each $f_k$ is continuous and $\lim\limits_{k\to\infty}f_k(x)=f(x)$ for all $x\in\mathbb R$.

We now prove that $f$ has the Darboux property. Let $I\subseteq \mathbb R$ be a connected set of cardinality larger than one. If $I\subseteq I_n$ for some $n\in\mathbb N$ then $f(I)$ is connected, provided the restriction $f|_{I_n}$ is continuous. Suppose $I\not\subseteq I_n$ for every $n\in\mathbb N$. Let $M=\{n\in\mathbb N: J_n=I_n\cap I\ne \emptyset\}$. Note that the set $G=\bigcup\limits _{n\in M}J_n$ is dense in $I$. Set $f(I)=U\cup V$, where $U$ and $V$ are disjoint clopen sets in $f(\mathbb R)$. Denote $K=\{n\in M:f(J_n)\subseteq U\}$ and $L=\{n\in M:f(J_n)\subseteq V\}$. Since the restriction of $f$ on each set $J_n$ is continuous, $G=G_1\cup G_2$, where $G_1=\bigcup\limits _{n\in K}J_n$, $G_2=\bigcup\limits _{n\in L}J_n$ and $G_1\cap G_2=\emptyset$. Lemma~\ref{pr:1.2} implies that $f(\overline{G_i})\subseteq \overline{f(G_i)}$ for $i=1,2$. Hence, $f(\overline{G_1})\subseteq U$ and $f(\overline{G_2})\subseteq V$. Therefore, $I=\overline{G_1}\cup \overline{G_2}$ and $\overline{G_1}\cap\overline{G_2}=\emptyset$. Consequently, $G_1=\emptyset$ or $G_2=\emptyset$. Thus, $U=\emptyset$ or $V=\emptyset$.

To show that $Y=f(\mathbb R)$ is hereditarily Lindel\"{o}f it is sufficient to prove that $Y_1=f(\mathbb R\setminus C)$ and $Y_2=f(C)$ are hereditarily Lindel\"{o}f. Note that $Y_1=f_0(\mathbb R)$ is hereditarily Lindel\"{o}f, since $Y_1$ is a continuous image of $\mathbb R$ under the continuous mapping $f_0$ with values in Hausdorff space  ${\mathcal F}(\mathbb R)$. Since $f(a_n)\cap [0,1]=\{\frac{1}{n}\}$ for every $n\in\mathbb N$ and $f(b)\cap [0,1]=\{0\}$ for each $b\in B$, the space $f(A)$ is countable discrete subspace of $Y_2$. Moreover, for each $b\in B$ the sets $f((b-\varepsilon,b]\cap C)$, where $\varepsilon>0$, form a base of neighborhoods of $f(x_0)$ in $Y_2$. Since an arbitrary union of sets of the form  $(u,v]$ is a union of a sequence $(u_n,v_n]$, $Y_2$ is hereditarily Lindel\"{o}f. Hence, $X$ is hereditarily Lindel\"{o}f, consequently, $X$ is perfectly normal.

Since $Y$ is perfectly normal and $f$ is a Baire-one function, $f$ is $F_\sigma$-measurable \cite[p.~394]{Ku1}. It remains to prove that the restriction $f|_C$ of $f$ on the Cantor set $C$ is everywhere discontinuous. Note that $f|_C$ is discontinuous everywhere on $A$, since $f(A)$ is discrete in $Y_2$. Moreover, for every $b\in B$ sets of the form $(b-\varepsilon,b]\cap C$ is not a neighborhood of $b$  in $C$. Therefore, $f|_C$ is discontinuous at each point $b\in B$.
\end{proof}

\end{document}